\documentclass{amsart}

\usepackage{amsfonts, amssymb, amsmath, eucal, verbatim, amsthm, amscd, enumerate}

\usepackage[dvipdfm]{graphicx,color}

\newtheorem{theorem}{Theorem}[section]
\newtheorem{lemma}[theorem]{Lemma}

\newtheorem{corollary}[theorem]{Corollary}
\theoremstyle{definition}

\theoremstyle{remark}
\newtheorem*{remark}{Remark}
\newtheorem*{remarks}{Remarks}

\newtheorem*{acknowledgement}{Acknowledgement}

\newtheorem*{organisation}{Organisation}

\numberwithin{equation}{section}
\def\R{{\mathbb R}}

\begin{document}

\date{11 November 2012}

\author{Neal Bez}
\address{Neal Bez, School of Mathematics, The Watson Building, University of Birmingham, Edgbaston,
Birmingham, B15 2TT, England} \email{n.bez@bham.ac.uk}
\author{Mitsuru Sugimoto}
\address{Graduate School of Mathematics, Nagoya University, Furocho, Chikusa-ku, Nagoya 464-8602, Japan}
\email{sugimoto@math.nagoya-u.ac.jp}
\begin{thanks} {The first author was supported by a JSPS Invitation Fellowship.}
\end{thanks}
\title{Optimal constants and extremisers for some smoothing estimates}
\maketitle

\begin{abstract}
We establish new results concerning the existence of extremisers for
a broad class of smoothing estimates of the form
$$
\|\psi(|\nabla|) \exp(it\phi(|\nabla|)f \|_{L^2(w)} \leq
 C\|f\|_{L^2},
$$
where the weight $w$ is radial and depends only on the spatial
variable; such a smoothing estimate is of course equivalent to the
$L^2$-boundedness of a certain oscillatory integral operator $S$
depending on $(w,\psi,\phi)$. Furthermore, when $w$ is homogeneous,
and for certain $(\psi,\phi)$, we provide an explicit spectral
decomposition of $S^*S$ and consequently recover an explicit formula
for the optimal constant $C$ and a characterisation of extremisers.
In certain well-studied cases when $w$ is inhomogeneous, we obtain
new expressions for the optimal constant.

\end{abstract}

\section{Introduction}
For real-valued functions $\Phi(\xi)$ and $\nabla=\nabla_x$, it is
easy to see that the solutions $u(t,x)=\exp(it\Phi(\nabla))f(x)$ to
the Cauchy problem of linear dispersive equations
\[
\left\{
\begin{aligned}
\left(i\partial_t+\Phi(\nabla)\right)\,u(t,x)&=0,\\
u(0,x)&=f(x)\in L^2(\R^d)
\end{aligned}
\right.
\]
preserve the $L^2$-norm of the initial data $f$, that is, we have
$\|u(t,\cdot)\|_{L^2_x(\R^d)}=\|f\|_{L^2(\R^d)}$ for any fixed time
$t\in\R$. But if we integrate the solution in $t$, we get an extra
gain of regularity in $x$. For example, we have the estimates
\begin{equation} \label{Schrodinger}
\left\|\Psi(x,\nabla) \exp(-it\Delta)f\right\|
_{L^2_{t,x}(\R\times\R^d)}
\le C\|f\|_{L^2(\R^d)}
\end{equation}
for the Schr\"odinger equation (the case $\Phi(\xi)=|\xi|^2$), where
\begin{equation*}
  \begin{array}{lllllllll}
  \mbox{[A]} &  \quad \Psi(x,\nabla)=(1+|x|^2)^{-1/2}(1-\Delta)^{1/4} &  \quad (d \geq 3), \\
  \mbox{[B]} &  \quad \Psi(x,\nabla)=|x|^{a-1}|\nabla|^{a} & \quad (a \in (1-\frac{d}{2},\frac{1}{2}),\, d \geq 2), \\
  \mbox{[C]} &  \quad \Psi(x,\nabla)=(1+|x|^2)^{-s/2}|\nabla|^{1/2} & \quad (s>\frac{1}{2},\,
d\geq2) .
  \end{array}
\end{equation*}
The estimate of type [A] is due to Kato and Yajima \cite{KatoYajima}
(see also \cite{BK}). Type [B] is due to Kato and Yajima
\cite{KatoYajima} for $a \in [0,\frac{1}{2})$ for $d \geq 3$, $a \in
(0,\frac{1}{2})$ for $d=2$ and Sugimoto \cite{Su1} for $a \in
(1-\frac{d}{2},\frac{1}{2})$ for all $d \geq 2$ (see also
\cite{BK}).
Type [C] is due to Kenig, Ponce and Vega \cite{KPV1}
(see also \cite{BK} and \cite{Ch}).

These estimates are often called smoothing estimates, and their
local version was first proved by Sj\"olin \cite{Sj}, Constantin and
Saut \cite{ConSau}, and Vega \cite{V}. There is a vast literature on
this subject, including Ben-Artzi and Devinatz \cite{BD1, BD2},
Hoshiro \cite{Ho1, Ho2}, Kenig, Ponce and Vega \cite{KPV1, KPV2,
KPV3, KPV4, KPV5, KPV6}, Linares and Ponce \cite{LP}, Sugimoto
\cite{Su2}, Walther \cite{WaltherSharp}, Ruzhansky and Sugimoto
\cite{RS}.

Rather less is known about the optimal constant for smoothing
estimates. In Simon \cite{Simon} and Watanabe \cite{Watanabe},
explicit optimal constants were given for type [B] smoothing
estimates. In significantly greater generality (under radial
assumptions on $\Phi$ and $\Psi$, and further mild conditions),
Walther \cite{WaltherBest} established an expression for the optimal
constant involving a double supremum; see Theorem \ref{t:Walther}
below. Our purpose in this paper is to provide a number of results
which build on these works, concerning both the optimal constant and
extremising initial data. Our results complement the recent body of
work concerning optimal Strichartz estimates; see, for example,
Christ and Shao \cite{ChristShao}, Fanelli, Vega and Visciglia
\cite{FVV, FVV2}, Foschi \cite{Foschi}, Ramos \cite{Ramos}, Bennett
\emph{et al.} \cite{BBCH}, Bez and Rogers \cite{BR}.

To each spatial dimension $d \geq 2$, radial weight $w : [0,\infty)
\to [0,\infty)$, smoothing function $\psi : [0,\infty) \to
[0,\infty)$, dispersion relation $\phi : [0,\infty) \to \mathbb{R}$,
and $f \in L^2(\mathbb{R}^d) \setminus \{0\}$, let
$\mathbf{C}_d(w,\psi,\phi;f)$ be the quantity given by
\begin{equation*}
 \mathbf{C}_d(w,\psi,\phi;f) =
 \frac{\| w(|x|)^{1/2} \psi(|\nabla|) \exp(it\phi(|\nabla|)f \|_{L^2_{t,x}(\mathbb{R} \times \mathbb{R}^d)}}{\|f\|_{L^2(\mathbb{R}^d)}}.
\end{equation*}
Of course,
\begin{equation} \label{e:optimalconstant}
 \mathbf{C}_d(w,\psi,\phi) = \sup_{f \in L^2(\mathbb{R}^d) \setminus \{0\}} \mathbf{C}_d(w,\psi,\phi;f)
\end{equation}
is the optimal constant $C \in (0,\infty]$ for which the smoothing
estimate
\begin{equation*}
\| w(|x|)^{1/2} \psi(|\nabla|) \exp(it\phi(|\nabla|)f \|_{L^2_{t,x}(\mathbb{R} \times \mathbb{R}^d)} \leq
 C\|f\|_{L^2(\mathbb{R}^d)}
\end{equation*}
holds for all $f \in L^2(\mathbb{R}^d)$.

Throughout the paper, we assume $(w,\psi,\phi)$ satisfies the basic
regularity condition that, for each $k \in \mathbb{N}_0$, the
function $\alpha_k : [0,\infty) \to [0,\infty)$ is continuous, where
\begin{equation*}
\alpha_k(\varrho) = \frac{\varrho \psi(\varrho)^2}{|\phi'(\varrho)|} \int_0^\infty J_{\nu(k)}(r \varrho)^2 r w(r) \, \mathrm{d}r.
\end{equation*}
Here, $J_{\nu}$ is the Bessel function of the first kind of order
$\nu$, and
$$
\nu(k) = \tfrac{d}{2} + k -1
$$
for each $k \in \mathbb{N}_0$. Implicitly, of course, this means
that we are assuming that $\phi$ is differentiable. We shall also
assume throughout the paper that $\phi$ is injective. Note that each
$\alpha_k$ is continuous if $w$ is integrable, $\psi$ is continuous
and $\phi$ is continuously differentiable; but as will become clear
we do not restrict ourselves to integrable weights.

\begin{theorem} \cite{WaltherBest} \label{t:Walther}
We have $\mathbf{C}_d(w,\psi,\phi) = \big(2\pi \sup_{k \in
\mathbb{N}_0} \sup_{\varrho \in [0,\infty)}
 \alpha_k(\varrho)\big)^{1/2}.$
\end{theorem}

We may define an equivalence relation $\approx$ on the set of
$(w,\psi,\phi)$ described above by
\begin{equation*}
(w,\psi,\phi) \approx (\widetilde{w},\widetilde\psi,\widetilde\phi) \qquad \text{if and only if}
\qquad w = \widetilde{w}, \quad \psi^2/|\phi'|=\widetilde\psi^2/|\widetilde\phi'|.
\end{equation*}
Clearly, by Theorem \ref{t:Walther}, we have that
\begin{equation} \label{e:comparison}
\mathbf{C}_d(w,\psi,\phi)=\mathbf{C}_d(\widetilde{w},\widetilde\psi,\widetilde\phi) \qquad \text{whenever} \qquad (w,\psi,\phi) \approx (\widetilde{w},\widetilde\psi,\widetilde\phi),
\end{equation}
so that the optimal constant is unchanged within each equivalence
class. We can also explain this fact by the {\it comparison
principle} discussed in Ruzhansky and Sugimoto \cite{RS}, where
non-radial functions $(w,\psi,\phi)$ are treated as well. All
explicit values of $\mathbf{C}_d(w,\psi,\phi)$ in the sequel are
given for the case $\phi(r) = r^2$ corresponding to the
Schr\"odinger equation. This is for simplicity and we emphasise that
further optimal constants are immediately available via
\eqref{e:comparison}.

Theorem \ref{t:Walther} leaves open several natural questions which
we shall address in this paper. Firstly, we shall consider the
existence and nature of extremisers for \eqref{e:optimalconstant};
that is, $f \in L^2(\mathbb{R}^d) \setminus \{0\}$ for which
\begin{equation*}
\mathbf{C}_d(w,\psi,\phi;f) = \mathbf{C}_d(w,\psi,\phi).
\end{equation*}
In order to state our first main result in this direction, Theorem
\ref{t:extremisers} below, let us introduce the notation
\begin{equation} \label{e:alpha}
\alpha = \sup_{k \in \mathbb{N}_0} \sup_{\varrho \in [0,\infty)} \alpha_k(\varrho).
\end{equation}
\begin{theorem} \label{t:extremisers}
An extremiser for \eqref{e:optimalconstant} exists if and only if
there exists $k_0 \in \mathbb{N}_0$ and a set $\mathcal{S} \subset
(0,\infty)$ of positive Lebesgue measure such that
$\alpha_{k_0}(\varrho) = \alpha$ for all $\varrho$ in $\mathcal{S}$.
\end{theorem}
We can, for example, deduce from Theorem \ref{t:extremisers} the
non-existence of extremisers for a broad class of smoothing
estimates for weights $w$ which are integrable. For this we will
establish the following.
\begin{theorem} \label{t:analytic}
Suppose $w \in L^1(0,\infty)$ and
  \begin{equation*}
  \varrho \mapsto \frac{\varrho \psi(\varrho)^2}{|\phi'(\varrho)|}
  \end{equation*}
  is real analytic on $(0,\infty)$. Then $\alpha_k$ is real analytic on $(0,\infty)$ for each $k \in
\mathbb{N}_0$.
\end{theorem}
As a sample application, by combining Theorems \ref{t:extremisers}
and \ref{t:analytic} we shall show the following.
\begin{corollary} \label{c:noextremisers}
  Suppose $w \in L^1(0,\infty)$ and
  \begin{equation*}
  \varrho \mapsto \frac{\varrho \psi(\varrho)^2}{|\phi'(\varrho)|}
  \end{equation*}
  is real analytic on $(0,\infty)$. If $\alpha_k$ is non-constant for each $k \in
  \mathbb{N}_0$, then there are no extremisers to
  \eqref{e:optimalconstant}. In particular, if  $w\not=0$ and
  \begin{equation} \label{e:ratioasymp}
  \frac{\psi(\varrho)^2}{|\phi'(\varrho)|}
  \end{equation}
  is asymptotically constant as $\varrho$ tends to zero and
 asymptotically nonzero constant as $\varrho$ tends to
infinity,
  then there are
  no extremisers to \eqref{e:optimalconstant}.
\end{corollary}
The hypotheses of Corollary \ref{c:noextremisers} are satisfied in
many classical smoothing estimates. For example, Simon showed in
\cite{Simon} that for the Schr\"odinger equation with
$(w(r),\psi(r),\phi(r)) = ((1+r^{2})^{-1},r^{1/2},r^2)$, we have
\begin{equation} \label{e:inhomoSimon}
\mathbf{C}_d(w,\psi,\phi) = (\pi/2)^{1/2}
\end{equation}
for each $d \geq 3$, that is, the optimal constant for smoothing
estimate of type [C] with $s=1$. Corollary \ref{c:noextremisers}
tells us immediately that \eqref{e:inhomoSimon} has no extremisers.

In \cite{Simon}, Simon further established that for
$(w(r),\psi(r),\phi(r)) = (r^{-2},1,r^2)$, we have
\begin{equation} \label{e:homoSimon}
\mathbf{C}_d(w,\psi,\phi) = (\pi/(d-2))^{1/2}
\end{equation}
for each $d \geq 3$, that is, the optimal constant for smoothing
estimate of type [B] with $a=0$. Of course, here the weight is not
integrable and we shall see that \emph{any} nonzero radial initial
data will be an extremiser. In fact, we provide a comprehensive
analysis of the case where the weight is radial and homogeneous. In
order to describe our results, it is convenient to let the linear
operator $S$ be given by\footnote{we have, of course, chosen to
suppress the dependence of $S$ on $w,\psi$ and $\phi$}
\begin{equation*}
Sf(x,t) = w(|x|)^{1/2} \int_{\mathbb{R}^d} \exp(i(x \cdot \xi +
t\phi(|\xi|)) \psi(|\xi|) f(\xi) \, \mathrm{d}\xi
\end{equation*}
for appropriate (say Schwartz) functions $f : \mathbb{R}^d \to
\mathbb{C}$. Note that
\begin{equation*}
S\widehat{f}(x,t) = (2\pi)^d w(|x|)^{1/2} \psi(|\nabla|) \exp(it\phi(|\nabla|)f(x),
\end{equation*}
where, $\widehat{f}$, the Fourier transform of $f$, is given by
\begin{equation*}
\widehat{f}(\xi) = \int_{\mathbb{R}^d} f(x) \exp(-i x \cdot \xi) \,\mathrm{d}x.
\end{equation*}
Therefore,
\begin{equation*}
\|S\| = (2\pi)^{d/2} \mathbf{C}_d(w,\psi,\phi),
\end{equation*}
where $\|S\|$ denotes the $L^2(\mathbb{R}^d) \to
L^2(\mathbb{R}^{d+1})$ operator norm of $S$. Our main result
concerning $S$ is the following.
\begin{theorem} \label{t:eigenfunctionT}
Let $(w(r),\psi(r),\phi(r)) = (r^{-2(1-a)},r^a,r^2)$, where $a \in
(1-\frac{d}{2},\frac{1}{2})$. For each $k \in \mathbb{N}_0$ we have
\begin{equation*}
S^*Sf(\eta) = 2^{d-1}\pi^{d+1/2} (-1)^k \frac{\Gamma(\tfrac{1}{2}-a)\Gamma(\tfrac{d}{2}+a-1)\Gamma(2-a-\tfrac{d}{2})}
{\Gamma(1-a)\Gamma(\tfrac{d}{2}-a+k)\Gamma(2-a-\frac{d}{2}-k)} f(\eta),
\end{equation*}
where
\begin{equation} \label{e:SHpiece}
f(\eta) = P(\eta) f_0(|\eta|) |\eta|^{-d/2 - k + 1/2}
\end{equation}
and $P$ is any solid spherical harmonic of degree $k$, and $f_0$ is
any element of $L^2(0,\infty)$. Consequently, the operator norm of
$S^*S$ is the largest eigenvalue
$$
2^{d-1}\pi^{d+1/2} \frac{\Gamma(\tfrac{1}{2}-a)\Gamma(\tfrac{d}{2}+a-1)}
{\Gamma(1-a)\Gamma(\tfrac{d}{2}-a)}
$$
and this is attained if and only if $S^*S$ is evaluated on any
radial function.
\end{theorem}
Underpinning Theorem \ref{t:eigenfunctionT} is the compactness of
the operator $L^2(\mathbb{S}^{d-1}) \to L^2(\mathbb{S}^{d-1})$ which
is the analogue of $S^*S$ restricted to $\mathbb{S}^{d-1}$. In
particular, with $(w,\psi,\phi)$ as in Theorem
\ref{t:eigenfunctionT}, let $T$ be the operator given by
\begin{equation} \label{e:Tdefn}
Tf(\eta) = |\eta|^a
\int_{\mathbb{R}^d} \frac{|\xi|^a}{|\xi - \eta|^{d+2a-2}}
\delta(|\xi|^2 - |\eta|^2)f(\xi) \, \mathrm{d}\xi
\end{equation}
and note that
\begin{equation} \label{e:T}
  T = \frac{1}{ 2\pi \gamma(d+2a-2)} S^*S,
\end{equation}
where
$$
\gamma(\lambda) = \frac{\pi^{d/2}2^\lambda \Gamma(\frac{1}{2}\lambda)}{\Gamma(\frac{1}{2}(d-\lambda))}.
$$
The identity \eqref{e:T} follows from the expression
$$
\widehat{\frac{1}{|\cdot|^{d-\lambda}}}(\xi) = \frac{\gamma(\lambda)}{|\xi|^\lambda}
$$
for the Fourier transform of a Riesz potential, valid for $\lambda
\in (0,d)$. Switching to polar coordinates, for $\eta \neq 0$, it
follows that
\begin{equation} \label{e:Tpolar}
Tf(\eta) = \frac{1}{2} \int_{\mathbb{S}^{d-1}} \frac{1}{|\theta - \eta'|^{d+2a-2}} f(|\eta|
\theta) \, \mathrm{d}\sigma(\theta),
\end{equation}
where $\eta' = |\eta|^{-1}\eta$.

We now define $T_{\mathbb{S}}$ to be the analogue of the operator
$T$ restricted to functions on the unit sphere, given by
\begin{equation*}
T_\mathbb{S} f(\omega) = \frac{1}{2} \int_{\mathbb{S}^{d-1}} \frac{1}{|\theta - \omega|^{d+2a-2}} f(\theta)
\,\mathrm{d}\sigma(\theta)
\end{equation*}
for each $f \in L^2(\mathbb{S}^{d-1})$.
\begin{theorem} \label{c:compact}
If $(w(r),\psi(r),\phi(r)) = (r^{-2(1-a)},r^a,r^2)$, where $a \in
(1-\frac{d}{2},\frac{1}{2})$, then the operator $T_\mathbb{S} :
L^2(\mathbb{S}^{d-1}) \to L^2(\mathbb{S}^{d-1})$ is compact. In
fact, if $k\in \mathbb{N}_0$ and $P$ is a solid spherical harmonic
of degree $k$, then $T_\mathbb{S} P = \lambda_k P$, where
$$
\lambda_k = \frac{\pi^{\frac{d-1}{2}}}{2^{2a}} \frac{(-1)^k\Gamma(\frac{1}{2}-a) \Gamma(2-a-\frac{d}{2})}
{\Gamma(2-a-\frac{d}{2}-k) \Gamma(-a+\frac{d}{2}+k)}.
$$
The sequence of eigenvalues $(\lambda_k)_{k \geq 0}$ is a decreasing
sequence converging to zero and hence the operator norm of
$T_\mathbb{S}$ is equal to
$$
\frac{\pi^{\frac{d-1}{2}}}{2^{2a}} \frac{\Gamma(\frac{1}{2}-a)}
{\Gamma(-a+\frac{d}{2})}.
$$
\end{theorem}

\begin{remark}
Theorems \ref{t:eigenfunctionT} and \ref{c:compact} have been stated
with each component of $(w,\psi,\phi)$ as a homogeneous function. It
is crucial to the proofs that $w$ is homogeneous; however, Theorems
\ref{t:eigenfunctionT} and \ref{c:compact} may be extended to $\psi$
and $\phi$ satisfying
\begin{equation*}
\psi(r)^2 = \lambda|\phi'(r)|r^{1-\mu},
\end{equation*}
where $w(r) = r^{-\mu}$, for some $\mu \in (1,d)$, and $\lambda$ is
some non-negative constant. In this case, the eigenvalues appearing
in Theorems \ref{t:eigenfunctionT} and \ref{c:compact} should be
multiplied by $2\lambda$. These facts will be clear from the
arguments in Section \ref{section:homo} and we omit the details.
\end{remark}

From Theorem \ref{t:eigenfunctionT} and the {\it duplication
formula}
\[
2^{2x-1}\Gamma(x)\Gamma(x+\tfrac{1}{2}) = \pi^{1/2}\Gamma(2x)
\qquad (x>0)
\]
(see \cite[p.240]{WW}), it follows that for $(w(r),\psi(r),\phi(r))
= (r^{-2(1-a)},r^a,r^2)$, we have
\begin{equation} \label{e:homoSimongeneral}
\mathbf{C}_d(w,\psi,\phi) = \bigg(\pi 2^{2a-1} \frac{\Gamma(1-2a) \Gamma(\frac{d}{2} +a -1)}{\Gamma(1-a)^2 \Gamma(\frac{d}{2} - a)} \bigg)^{1/2}
\end{equation}
for each $d \geq 2$ and each $a \in (1-\frac{d}{2},\frac{1}{2})$,
that is, the optimal constant for smoothing estimate of type [B]
with general $a$, and
\begin{equation*}
\mathbf{C}_d(w,\psi,\phi) = \mathbf{C}_d(w,\psi,\phi;f)
\end{equation*}
precisely when $f \in L^2(\mathbb{R}^d)$ is radial. The case $a=0$
and $d \geq 3$ is the optimal constant in \eqref{e:homoSimon} due to
Simon.

Our argument leading to Theorem \ref{t:eigenfunctionT} essentially
proceeds by multiplying out the $L^2(\mathbb{R}^{d+1})$ norm of
$Sf$, an idea which has been fruitful on several occasions in
understanding Lebesgue space norms of oscillatory integral operators
when the exponent is an even integer. In this particular case of
$(w,\psi,\phi)$, this approach is different to (and more
straightforward) than the approach of Walther in proving Theorem
\ref{t:Walther}. We note, however, that in earlier work, Watanabe
\cite{Watanabe} (see also \cite{Chen}) used the multiplying out
approach to show that radial input functions are extremisers in the
homogeneous case $(w(r),\psi(r),\phi(r)) = (r^{-2(1-a)},r^a,r^2)$,
and gave an expression of the optimal constant. One should view
Theorems \ref{t:eigenfunctionT} and \ref{c:compact} as extensions of
this result in \cite{Watanabe}. We mention a different extension in
very recent work of Ozawa and Rogers \cite{OzawaRogers} where the
sharp Hardy--Littlewood--Sobolev inequality on the sphere, due to
Lieb, is used to establish certain angular refinements with optimal
constants and characterisations of extremisers.

Our final contribution in this paper is to explicitly compute the
quantity $\alpha$ in \eqref{e:alpha} (and hence the optimal constant
in the associated smoothing estimate) in certain cases where the
weight is inhomogeneous. The finiteness of
$\mathbf{C}_d(w,\psi,\phi)$ when
\begin{equation} \label{e:Simonopen}
(w(r),\psi(r),\phi(r)) = ((1+r^2)^{-1},(1+r^2)^{1/4},r^2)
\end{equation}
and
$d \geq 3$, that is, the smoothing estimate of type [A],
motivated the considerations of optimal constants for smoothing
estimates by Simon in \cite{Simon}, which led to
\eqref{e:inhomoSimon} and \eqref{e:homoSimon}. However, the value of
$\mathbf{C}_d(w,\psi,\phi)$ for $(w,\psi,\phi)$ in
\eqref{e:Simonopen} was left open in \cite{Simon}. We compute the
value of $\alpha$, and hence $\mathbf{C}_d(w,\psi,\phi)$, in this
case, and the closely related case where
$$
(w(r),\psi(r),\phi(r)) = ((1+r^2)^{-1},(1+r)^{1/2},r^2),
$$
in spatial dimensions $d=3$ and $d=5$.
\begin{theorem} \label{t:inhomoconstants} If $(w(r),\psi(r),\phi(r)) =
((1+r^2)^{-1},(1+r^2)^{1/4},r^2)$ then
\begin{equation}
\mathbf{C}_3(w,\psi,\phi) =
\pi^{1/2} \quad \text{and} \quad \mathbf{C}_5(w,\psi,\phi) = (\pi/2)^{1/2}.
\end{equation}
If $(w(r),\psi(r),\phi(r)) = ((1+r^2)^{-1},(1+r)^{1/2},r^2)$ then
\begin{equation}
\mathbf{C}_3(w,\psi,\phi) =
\pi^{1/2} \quad \text{and} \quad \mathbf{C}_5(w,\psi,\phi) = (2\pi\alpha_0(\varrho_0))^{1/2},
\end{equation}
where $\varrho_0$ is the unique positive solution of
$$
(3+2\varrho+2\varrho^2+\varrho^3)\sinh \varrho = \varrho(3+2\varrho+\varrho^2)\cosh \varrho.
$$
\end{theorem}
Key to our proof of Theorem \ref{t:inhomoconstants} is the
monotonicity of certain quantities involving modified Bessel
functions of the first kind, $I_{\nu}(\varrho)$ and
$K_{\nu}(\varrho)$. We will use monotonicity properties in both the
argument $\varrho$ and the index $\nu$.

We remark that the optimal constants for smoothing estimates of type
[A] with $d=3,5$, type [B], and type [C] with $s=1$ have been thus
explicitly determined, but those for other cases are still left
open.

In all of the above cases where we have found the optimal constant
(including the case of homogeneous weights in
\eqref{e:homoSimongeneral}), it is true that
\begin{equation} \label{e:conj}
  \mathbf{C}_d(w,\psi,\phi) = \big(2\pi \sup_{\varrho \in [0,\infty)} \alpha_0(\varrho)\big)^{1/2};
\end{equation}
that is, the supremum in $k \in \mathbb{N}_0$ in \eqref{e:alpha} is
attained at $k=0$. We shall see that the supremum in $\varrho$ may
be attained in several ways; see the remarks at the end of Section
\ref{section:inhomo}.

It is conceivable that one could find a geometric characterisation
of the $(w,\psi,\phi)$ under which \eqref{e:conj} is true. This is
suggested by earlier work of several authors in the case of weighted
$L^2$ estimates for solutions of the Helmholtz equation, or weighted
$L^2$ estimates for the Fourier extension operator associated to the
unit sphere, where boundedness is known to be equivalent to the
$L^\infty$-boundedness of an $X$-ray transform applied to the weight
$w$; see, for example, \cite{BBC}, \cite{BRV}, \cite{CS}, \cite{M}.
This viewpoint led to the simple example of $(w,\psi,\phi)$ at the
end of Section \ref{section:inhomo} where \eqref{e:conj} fails. In
this example, the weight is supported \emph{away from the origin},
unlike the weights of the form $w(r) = r^{-\lambda}$, $w(r) =
(1+r^2)^{-\lambda/2}$ or $w(r) = (1+r)^{-\lambda}$ considered above
for which \eqref{e:conj} holds.

\begin{organisation}
In the subsequent section, we introduce some notation and facts
concerning spherical harmonics and Bessel functions of the first
kind. In Section \ref{section:extremisers} we prove Theorems
\ref{t:extremisers} and \ref{t:analytic}. Section \ref{section:homo}
is concerned with the case of homogeneous weights where Theorems
\ref{t:eigenfunctionT} and \ref{c:compact} are proved and several
further remarks are given. Finally, in Section \ref{section:inhomo}
we prove Theorem \ref{t:inhomoconstants}.
\end{organisation}
\begin{acknowledgement}
The first author would like to thank Franck Barthe, Jon Bennett and
Keith Rogers for very useful conversations.
\end{acknowledgement}

\section{Preliminaries and notation}

The notation $A \lesssim_{p_1,\ldots,p_m} B$ means that $A \leq CB$,
where the constant $C$ depends on at most the parameters
$p_1,\ldots,p_m$. Also, $A \sim_{p_1,\ldots,p_m} B$ means $A
\lesssim_{p_1,\ldots,p_m} B$ and $B \lesssim_{p_1,\ldots,p_m} A$.

We use $\mathrm{d}\sigma$ throughout as the induced Lebesgue measure
on the unit sphere $\mathbb{S}^{d-1}$ of $\mathbb{R}^d$.

\subsection{Spherical harmonic decomposition of $L^2(\mathbb{R}^d)$}
Let $k \in \mathbb{N}_0$. Write $\mathfrak{A}_k$ for the space of
solid spherical harmonics; that is, the space of polynomials on
$\mathbb{R}^d$ with complex coefficients which are homogeneous of
degree $k$ and harmonic. Also, we let $\mathfrak{H}_k$ denote the
space of all linear combinations of functions of the form
\begin{equation*}
\xi \mapsto P(\xi) f_0(|\xi|) |\xi|^{-d/2-k+1/2}
\end{equation*}
where $P \in \mathfrak{A}_k$ and $f_0 \in L^2(0,\infty)$. It will be
convenient to fix an orthonormal basis
$\{P^{(k,1)},\ldots,P^{(k,a_k)}\}$ of $\mathfrak{A}_k$, so that each
$f \in \mathfrak{H}_k$ may be written
\begin{equation*}
f(\xi) = \sum_{m=1}^{a_k} P^{(k,m)}(\xi)f_0^{(m)}(|\xi|)|\xi|^{-d/2-k+1/2},
\end{equation*}
where $f_0^{(m)} \in L^2(0,\infty)$, $1 \leq m \leq a_k$.

We shall use the complete orthogonal direct sum decomposition
\begin{equation} \label{e:directsum}
L^2(\mathbb{R}^d) = \bigoplus_{k=0}^\infty \mathfrak{H}_k
\end{equation}
in the sense that the $\mathfrak{H}_k$ are closed mutually
orthogonal subspaces of $L^2(\mathbb{R}^d)$, and each $f \in
L^2(\mathbb{R}^d)$ may be expressed as $\sum_{k=1}^\infty f_k$ where
$f_k \in \mathfrak{H}_k$ for each $k \in \mathbb{N}_0$. We refer the
reader to \cite{SteinWeiss} for further details.

\subsection{Properties of the Bessel function $J_\nu$}
For $\text{Re} \,\nu > -\frac{1}{2}$ and $z \in \mathbb{C}$ such
that $\text{arg}\,(z) \in (-\pi,\pi)$, the Bessel function $J_\nu$
is given by the expression
\begin{equation} \label{e:Besseldefn}
J_\nu(z) = \frac{(z/2)^\nu}{\Gamma(\frac{1}{2})\Gamma(\nu + \frac{1}{2})} \int_{-1}^1 e^{izt} (1-t^2)^{\nu - \frac{1}{2}} \, \mathrm{d}t.
\end{equation}
Mostly we are concerned with $J_\nu(r)$ when $r \in [0,\infty)$. For
$\nu \in \tfrac{1}{2}\mathbb{N}$ it is well-known that explicit
formulae in terms of elementary functions for $J_\nu$ are available;
for example,
\begin{equation} \label{e:Jhalf}
J_{1/2}(r) = (\tfrac{2}{\pi r})^{1/2} \sin (r),
\quad \text{and} \quad J_{3/2}(r) = (\tfrac{2}{\pi r})^{1/2} (\tfrac{\sin(r)}{r} - \cos(r)),
\end{equation}
which we need on several occasions.

We conclude this section with two asymptotic results concerning
$J_\nu$.
\begin{theorem} \label{t:besselasymp}
Suppose $\nu > -\frac{1}{2}$. Then
\begin{equation*}
|J_\nu(r) - (\tfrac{2}{\pi r})^{1/2} \cos(r - \tfrac{\pi}{2}\nu  - \tfrac{\pi}{4})| \lesssim_\nu r^{-3/2}
\end{equation*}
for all $r \geq 1$.
\end{theorem}
For a proof of Theorem \ref{t:besselasymp}, see \cite{SteinWeiss}.
\begin{theorem} \label{t:limits}
Suppose $w \in L^1(0,\infty)$ and $\nu > -\frac{1}{2}$. Then
$$
\varrho \int_0^\infty J_\nu(r \varrho)^2 rw(r) \, \mathrm{d}r
$$
tends to zero as $\varrho$ tends to zero, and tends to
$\frac{1}{\pi}\|w\|_{L^1(0,\infty)}$ as $\varrho$ tends to infinity.
\end{theorem}

\begin{proof}
We use the Bessel function asymptotics in Theorem
\ref{t:besselasymp}. In particular, it follows that for all $r > 0$
we have
\begin{equation*}
r^{1/2} J_{\nu(k)}(r) = (\tfrac{2}{\pi})^{1/2} \cos(r - \ell) + E(r),
\end{equation*}
where
\begin{equation} \label{e:error}
|E(r)| \lesssim_{d,k} (1+r)^{-1}
\end{equation}
and $\ell = \frac{\pi}{2}\nu(k) + \frac{\pi}{4}$. Therefore
\begin{equation} \label{e:rJ(r)^2}
rJ_{\nu(k)}(r)^2 = \tfrac{2}{\pi} \cos^2(r-\ell) + \widetilde{E}(r)
\end{equation}
for all $r > 0$. Here $\widetilde{E}$ also satisfies an estimate of
the form \eqref{e:error} and therefore
\begin{equation*}
\int_0^\infty \widetilde{E}(r \varrho) w(r) \, \mathrm{d}r \rightarrow 0 \qquad \text{as $\varrho
\to \infty$}
\end{equation*}
by the dominated convergence theorem and since $w \in
L^1(0,\infty)$. For the main term we have
\begin{align*}
& 4\int_0^\infty \cos^2(r \varrho-\ell) w(r) \, \mathrm{d}r \\
& = e^{-2\ell i} \int_0^\infty e^{2i r \varrho}
w(r) \, \mathrm{d}r + e^{2\ell i} \int_0^\infty e^{-2i r \varrho}
w(r) \, \mathrm{d}r  + 2\int_0^\infty w(r) \, \mathrm{d}r.
\end{align*}
The first two terms on the right-hand side tend to zero as $\varrho$
tends to infinity by the Riemann--Lebesgue lemma, again using $w \in
L^1(0,\infty)$. Hence
\begin{equation*}
\int_0^\infty \cos^2(r \varrho-\ell) w(r) \, \mathrm{d}r \to \tfrac{1}{2}\|w\|_{L^1(0,\infty)} \qquad \text{as $\varrho
\to \infty$}
\end{equation*}
and it follows that
$$
\varrho \int_0^\infty J_\nu(r \varrho)^2 rw(r) \, \mathrm{d}r \to \tfrac{1}{\pi}\|w\|_{L^1(0,\infty)}
$$
as $\varrho \to \infty$ as claimed.

Also, note that $r \varrho J_{\nu(k)}(r \varrho)^2 \lesssim_{d,k} 1$
uniformly in $\varrho > 0$ by \eqref{e:rJ(r)^2}, and it follows
immediately from the dominated convergence theorem and the
boundedness of the Bessel function that
$$
\varrho \int_0^\infty J_\nu(r \varrho)^2 rw(r) \, \mathrm{d}r \to 0
$$
as $\varrho \to 0$.
\end{proof}

\section{Extremisers: Proofs of Theorems \ref{t:extremisers} and
\ref{t:analytic}} \label{section:extremisers}

Theorem \ref{t:extremisers} will follow from a re-visit of Walther's
proof of Theorem \ref{t:Walther} discussed in \cite{WaltherBest},
which we briefly recall now. The first step is an application of
Plancherel in time for each fixed $x \in \mathbb{R}^d$. To see this
explicitly, first note that
\begin{equation*}
Sf(x,t) = \int_\mathbb{R} \exp(it\varrho) \widetilde{f}[x](\varrho) \, \mathrm{d}\varrho,
\end{equation*}
where
\begin{equation*}
\widetilde{f}[x](\varrho) = \frac{w(|x|)^{1/2} \psi(\phi^{-1}(\varrho))[\phi^{-1}(\varrho)]^{d-1} }{|\phi'(\phi^{-1}(\varrho))|} \int_{\mathbb{S}^{d-1}} \exp(i\phi^{-1}(\varrho)x \cdot \theta) f(\phi^{-1}(\varrho) \theta) \, \mathrm{d}\sigma(\theta)
\end{equation*}
for $\varrho \in \phi((0,\infty))$, and $\widetilde{f}[x](\varrho) =
0$ otherwise. Therefore,
\begin{equation*}
\| Sf \|_{L^2_{t,x}(\mathbb{R} \times \mathbb{R}^d)}^2 = 2\pi  \|\widetilde{f}\|_{L^2_{\varrho,x}(\mathbb{R} \times \mathbb{R}^d)}^2.
\end{equation*}
Orthogonality considerations (see \cite[Sect. 4.2.3]{WaltherSharp})
lead to
\begin{equation*}
\| Sf \|_{L^2_{t,x}(\mathbb{R} \times \mathbb{R}^d)}^2 = 2\pi \sum_{k =0}^\infty \sum_{m=1}^{a_k}
\|\widetilde{f^{(k,m)}} \|_{L^2_{\varrho,x}(\mathbb{R} \times \mathbb{R}^d)}^2,
\end{equation*}
where $f = \sum_{k=0}^\infty \sum_{m=1}^{a_k} f^{(k,m)}$ and
\begin{equation} \label{e:fkmdefn}
f^{(k,m)}(\xi) = P^{(k,m)}(\xi) f^{(k,m)}_0(|\xi|)|\xi|^{-d/2-k+1/2}
\end{equation}
for some $f^{(k,m)}_0 \in L^2(0,\infty)$.

If $k \in \mathbb{N}_0$ and $P \in \mathfrak{A}_k$ then we have
\begin{equation} \label{e:Psigmahat}
  \widehat{P \mathrm{d}\sigma}(x) = \frac{(2\pi)^{d/2}}{i^k}
P(x)J_{\nu(k)}(|x|)|x|^{-\nu(k)}
\end{equation}
for each $x \in \mathbb{R}^d$. From this and certain changes of
variables we obtain
\[
 \|\widetilde{f^{(k,m)}} \|_{L^2_{\varrho,x}(\mathbb{R} \times \mathbb{R}^d)}^2 =
(2\pi)^{d}\int_0^\infty \alpha_k(\varrho)
|f^{(k,m)}_0(\varrho)|^2 \, \mathrm{d}\varrho
\]
(see \cite[Sect 6.1]{WaltherBest}) and consequently,
\begin{align}
\| Sf \|_{L^2_{t,x}(\mathbb{R} \times \mathbb{R}^d)}^2
& = (2\pi)^{d+1} \sum_{k=0}^\infty \sum_{m=1}^{a_k} \int_0^\infty \alpha_k(\varrho)
|f^{(k,m)}_0(\varrho)|^2 \, \mathrm{d}\varrho \label{e:mainWaltheridentity} \\
& \leq (2\pi)^{d+1} \sum_{k=0}^\infty \sup_{\varrho > 0} \alpha_k(\varrho) \sum_{m=1}^{a_k} \int_0^\infty
|f^{(k,m)}_0(\varrho)|^2 \, \mathrm{d}\varrho \label{e:Walther1} \\
& \leq (2\pi)^{d+1} \alpha \sum_{k=0}^\infty \sum_{m=1}^{a_k} \int_0^\infty
|f^{(k,m)}_0(\varrho)|^2 \, \mathrm{d}\varrho = (2\pi)^{d+1} \alpha \|f\|_{L^2(\mathbb{R}^d)}^2. \label{e:Walther2}
\end{align}

Let us see that the constant $(2\pi)^{d+1} \alpha$ in the above
estimate is optimal, given that each $\alpha_k$ is
continuous\footnote{such considerations are not explicitly included
in \cite{WaltherBest}}. To begin, let $\varepsilon
> 0$. Then there exist $k_0 \in \mathbb{N}_0$ and $\varrho_0 > 0$
such that $\alpha - 2\varepsilon < \alpha_{k_0}(\varrho_0) \leq
\alpha$ and by continuity there exists $\delta > 0$ such that
$\alpha - \varepsilon < \alpha_{k_0}(\varrho) \leq \alpha$ for each
$\varrho \in [\varrho_0 - \delta,\varrho_0 + \delta]$. Now let $f
\in \mathfrak{H}_{k_0}$ be given by
$$
f(\xi) = P(\xi) f_0(|\xi|) |\xi|^{-d/2 - k_0 + 1/2},
$$
where $P$ is any element of $\mathfrak{A}_{k_0}$ normalised so that
$\|P\|_{L^2(\mathbb{S}^{d-1})} = 1$ and $f_0$ is any nonzero element
of $L^2(0,\infty)$ which is supported on $[\varrho_0 -
\delta,\varrho_0 + \delta]$. Using equality
\eqref{e:mainWaltheridentity} we get
\begin{align*}
\| Sf \|_{L^2_{t,x}(\mathbb{R} \times \mathbb{R}^d)}^2 & = (2\pi)^{d+1} \int_0^\infty \alpha_{k_0}(\varrho) |f_0(\varrho)|^2 \,\mathrm{d}\varrho \\
& \geq (2\pi)^{d+1} (\alpha - \varepsilon) \|f_0\|_{L^2(0,\infty)}^2 =  (2\pi)^{d+1} (\alpha - \varepsilon) \|f\|_{L^2(\mathbb{R}^d)}^2,
\end{align*}
and consequently the constant $(2\pi)^{d+1}\alpha$ cannot be
bettered.

\begin{proof}[Proof of Theorem \ref{t:extremisers}]
Suppose $f \in L^2(\mathbb{R}^d) \setminus \{0\}$ satisfies
\begin{equation} \label{e:extremiser}
  \mathbf{C}_d(w,\psi,\phi) = \mathbf{C}_d(w,\psi,\phi;f)
\end{equation}
so that the inequalities in \eqref{e:Walther1} and
\eqref{e:Walther2} are both equalities. As above, we write $f =
\sum_{k = 0}^\infty f_k$, where $f_k = \sum_{m=1}^{a_k} f^{(k,m)}$
and $f^{(k,m)}$ is given by \eqref{e:fkmdefn}. Let $F^{(k)}(\varrho)
= \sum_{m=1}^{a_k} |f^{(k,m)}_0(\varrho)|^2$ so that
$$
\int_0^\infty F^{(k)}(\varrho) \, \mathrm{d}\varrho = \|f_k\|_{L^2(\mathbb{R}^d)}^2.
$$
Also, let
$$
\mathcal{K} = \{ k \in \mathbb{N}_0 : \sup_{\varrho > 0} \alpha_k(\varrho) = \alpha \}.
$$
From equality in \eqref{e:Walther2}, it follows that $f_k$ must be
zero for $k \notin \mathcal{K}$. So $f = \sum_{k \in \mathcal{K}}
f_k$ and since $f \neq 0$ there exists $k_0 \in \mathcal{K}$ such
that $f_{k_0} \neq 0$. From equality in \eqref{e:Walther1}, we see
that for all $k \in \mathcal{K}$ we must have
$$
\alpha_k(\varrho) = \alpha \qquad \text{for all $\varrho \in \mbox{supp}\, F^{(k)}$.}
$$
Now $F^{(k_0)} \in L^1(0,\infty) \setminus \{0\}$ and hence the
desired conclusion holds by taking $\mathcal{S} =
\mbox{supp}\,F^{(k_0)}$.

For the converse, suppose we are given a set $\mathcal{S}$ of
positive Lebesgue measure and $k_0 \in \mathbb{N}_0$ such that
$\alpha_{k_0}(\varrho) = \alpha$ for each $\varrho \in \mathcal{S}$.
Let $f \in L^2(\mathbb{R}^d) \setminus \{0\}$ be given by
$$
f(\xi) = P(\xi) f_0(|\xi|) |\xi|^{-d/2 - k_0 + 1/2},
$$
where $P$ is any element of $\mathfrak{A}_{k_0}$ normalised so that
$\|P\|_{L^2(\mathbb{S}^{d-1})} = 1$ and $f_0$ is any nonzero
function in $L^2(0,\infty)$ which is supported on $\mathcal{S}$.
Then it is clear from \eqref{e:mainWaltheridentity} that we have
equality in both \eqref{e:Walther1} and \eqref{e:Walther2} and hence
\eqref{e:extremiser} holds for such $f$.
\end{proof}

\begin{proof}[Proof of Theorem \ref{t:analytic}]
It is clearly enough to prove that $\widetilde{\alpha} :
\mathfrak{S} \to \mathbb{C}$ is complex analytic on the strip
$\mathfrak{S}$, where
$$
\widetilde{\alpha}(z) = \int_0^\infty J_{\nu}(rz)^2rw(r)\,\mathrm{d}r
$$
and
\begin{equation*}
  \mathfrak{S} = \{ z \in \mathbb{C} : \text{Re}(z) > 0 \,\,\, \text{and} \,\,\, \text{Im}(z) \in (-1,1)\}.
\end{equation*}
Here, $J_{\nu}$ denotes the usual analytic extension of the Bessel
function to the half-plane $\{z \in \mathbb{C} : \text{Re}(z) >
0\}$, given by \eqref{e:Besseldefn}, and $\nu \geq 0$ is fixed. To
this end, for each $N \in \mathbb{N}$, let $\widetilde{\alpha}_N :
\mathfrak{S} \to \mathbb{C}$ be given by
$$
\widetilde{\alpha}_N(z) = \int_0^N J_{\nu}(rz)^2rw(r)\,\mathrm{d}r,
$$
for $z \in \mathfrak{S}$. We claim that each $\widetilde{\alpha}_N$
is complex analytic on $\mathfrak{S}$ and $\widetilde{\alpha}_N$
converges uniformly to $\widetilde{\alpha}$ on every compact subset
of $\mathfrak{S}$. From the claim, it follows that
$\widetilde{\alpha}$ is complex analytic on $\mathfrak{S}$ as
required.

To see that our claim is true, let $\mathfrak{D} \subset
\mathfrak{S}$ be compact and note that $|\text{Re}(z)| \geq
\varepsilon$, for all $z \in \mathfrak{D}$, where $\varepsilon$ is
some strictly positive constant depending on $\mathfrak{D}$. From
Theorem \ref{t:besselasymp} (see also Watson \cite{Watson}, page
199) it follows that
$$
|J_{\nu}(z)| \lesssim_{\nu} (1 + |\text{Re}(z)|)^{-1/2}
$$
for each $z \in \mathfrak{D}$, and therefore,
\begin{equation*}
  |\widetilde{\alpha}(z) - \widetilde{\alpha}_N(z)| \lesssim_{\nu} \int_N^\infty \frac{rw(r)}{1+r|\text{Re}(z)|} \, \mathrm{d}r \lesssim_{\varepsilon,\nu} \int_N^\infty
  w(r) \,\mathrm{d}r.
\end{equation*}
Hence, $\sup_{z \in \mathfrak{D}} |\widetilde{\alpha}(z) -
\widetilde{\alpha}_N(z)| \to 0$ uniformly as $N \to \infty$ as
required.

Finally, a straightforward argument using the complex analyticity
and boundedness properties of $J_\nu$ on $\mathfrak{S}$, shows that
each $\widetilde{\alpha}_N$ is complex analytic on $\mathfrak{S}$.
This completes the proof of our claim, and hence Theorem
\ref{t:analytic}.
\end{proof}

\begin{proof}[Proof of Corollary \ref{c:noextremisers}]
Fix $k \in \mathbb{N}_0$. Since $\alpha_k$ is analytic, the
pre-image set $\alpha_k^{-1}(\alpha)$ is either equal to
$(0,\infty)$ or it has Lebesgue measure zero. This follows because
the zero set of an analytic function on $(0,\infty)$ is either
$(0,\infty)$ or contains only isolated points. In the latter case,
the zero set is countable and hence has Lebesgue measure zero.
Hence, by Theorem \ref{t:extremisers}, if each $\alpha_k$ is
non-constant then no extremisers exist.

From Theorem \ref{t:limits} and our hypotheses on the ratio in
\eqref{e:ratioasymp}, we know that $\alpha_k(\varrho) \to 0$ as
$\varrho \to 0$ and $\alpha_k(\varrho)$ tends to a strictly positive
number as $\varrho \to \infty$. This means each $\alpha_k$ is not
constant and therefore no extremisers exist.
\end{proof}

\section{Homogeneous weights: Proof of Theorems
\ref{t:eigenfunctionT} and \ref{c:compact}} \label{section:homo}

Let $(w(r),\psi(r),\phi(r)) = (r^{-2(1-a)},r^a,r^2)$ where $a \in
(1-\tfrac{d}{2},\tfrac{1}{2})$ and $d \geq 2$. We first prove
Theorem \ref{c:compact} concerning $T_\mathbb{S}$, which we recall
is given by
\begin{equation*}
T_{\mathbb{S}}f(\omega) = \frac{1}{2} \int_{\mathbb{S}^{d-1}} \frac{1}{|\theta - \omega|^{d+2a-2}} f(\theta) \, \mathrm{d}\sigma(\theta).
\end{equation*}
We remark that if $f$ is constant then the rotation invariance of
$\mathrm{d}\sigma$ clearly implies that $f$ is an eigenvector of
$T_\mathbb{S}$ with an explicitly computable eigenvalue. In order to
extend this to the full strength of Theorem \ref{c:compact}, we use
the Funk--Hecke theorem.
\begin{theorem}[Funk--Hecke] \label{t:FunkHecke}
Let $k \in \mathbb{N}_0$ and let $P$ be a spherical harmonic of
degree $k$. Then, for each unit vector $\omega$,
\begin{equation*}
\int_{\mathbb{S}^{d-1}} F(\omega \cdot \theta)P(\theta) \, \mathrm{d}\sigma(\theta) = P(\omega)
\frac{|\mathbb{S}^{d-2}|}{C_{d,k}(1)} \int_{-1}^1 F(t) C_{d,k}(t)
(1-t^2)^{\frac{d-3}{2}} \,\mathrm{d}t
\end{equation*}
holds whenever the complex-valued function $F$ is integrable on
$[-1,1]$ with respect to the weighted Lebesgue measure
$(1-t^2)^{\frac{d-3}{2}}\,\mathrm{d}t.$
\end{theorem}
Here, $C_{d,k}$ is the Gegenbauer (or ultraspherical) polynomial of
degree $k$ associated with $\frac{d-2}{2}$, defined via the
generating function
\begin{equation*}
(1 - 2st + t^2)^{-\frac{d-2}{2}} = \sum_{k = 0}^
\infty C_{d,k}(s)t^k
\end{equation*}
for $|s| \leq 1$ and $|t| < 1$ (see, for example,
\cite{SteinWeiss}). For a proof of the Funk--Hecke theorem, see
\cite{Seeley}.

\begin{proof}[Proof of Theorem \ref{c:compact}]
Let $P$ be a spherical harmonic of degree $k$ and note that
\begin{align*}
T_\mathbb{S}P(\omega) & = \frac{1}{2} \int_{\mathbb{S}^{d-1}} \frac{1}{|\theta - \omega|^{d+2a-2}} \,P(\theta) \, \mathrm{d}\sigma(\theta) \\
& = 2^{-\frac{d+2a}{2}} \int_{\mathbb{S}^{d-1}}
\frac{1}{(1-\theta \cdot \omega)^{\frac{d+2a-2}{2}}} \,P(\theta) \, \mathrm{d}\sigma(\theta),
\end{align*}
and thus, by the Funk--Hecke Theorem,
\begin{align} \label{e:afterFH}
T_\mathbb{S}P(\omega) = P(\omega)\, 2^{-\frac{d+2a}{2}} \frac{|\mathbb{S}^{d-2}|}{C_{d,k}(1)} \int_{-1}^1 (1-t)^{-\frac{d+2a-2}{2}} C_{d,k}(t)
(1-t^2)^{\frac{d-3}{2}} \,\mathrm{d}t.
\end{align}
We have
$$
C_{d,k}(1) = \frac{\Gamma(d-2+k)}{k!\Gamma(d-2)},
$$
which can be found in \cite{Shimakura}, and therefore, using the
formula in terms of the Gamma function for the integral in
\eqref{e:afterFH} from \cite{GR} (page 795), we obtain
\begin{align*}
& \frac{|\mathbb{S}^{d-2}|}{C_{d,k}(1)} \int_{-1}^1 (1-t)^{-\frac{d+2a-2}{2}} C_{d,k}(t)
(1-t^2)^{\frac{d-3}{2}} \,\mathrm{d}t \\ & =
\frac{|\mathbb{S}^{d-2}|}{C_{d,k}(1)} \int_{-1}^1 (1-t)^{-\frac{1+2a}{2}} (1+t)^{\frac{d-3}{2}} C_{d,k}(t)
 \,\mathrm{d}t \\
& = (-1)^k 2^{\frac{d-2a}{2}} \pi^{\frac{d-1}{2}} \frac{\Gamma(\frac{1}{2}-a)
\Gamma(2-a-\frac{d}{2})}
{\Gamma(2-a-\frac{d}{2}-k)\Gamma(-a+\frac{d}{2}+k)}
\end{align*}
which is equal to $\lambda_k$. Hence, $T_\mathbb{S} P = \lambda_k
P$, as claimed.

Since
$$
\frac{\lambda_k}{\lambda_{k+1}} = \frac{-a+\frac{d}{2}+k}{a-1+\frac{d}{2} + k}
$$
is strictly larger than one for $a \in (1-\frac{d}{2},\frac{1}{2})$,
it follows that $(\lambda_k)_{k \geq 0}$ is a decreasing sequence.
To complete the proof of Theorem \ref{c:compact} it remains to show
that $\lambda_k \to 0$ as $k \to \infty$. For this, in the case
$a+\frac{d}{2} \notin \mathbb{Z}$ we have
$$
\Gamma(2-a-\tfrac{d}{2}-k) \Gamma(-a+\tfrac{d}{2}+k) =
\Gamma(1-s)\Gamma(t+s),
$$
where
\begin{equation*}
  s = -1 + a + \frac{d}{2} + k \qquad \text{and} \qquad t = 1-2a.
\end{equation*}
By the Euler reflection formula (using that $s \notin \mathbb{Z}$),
\begin{equation*}
  \Gamma(1-s)\Gamma(s) = \frac{\pi}{\sin (\pi s)},
\end{equation*}
and therefore
\begin{equation} \label{e:snotinteger}
 \bigg|\frac{1}{\Gamma(1-s)\Gamma(t+s)}\bigg| \leq \frac{\Gamma(s)}{\Gamma(t+s)}.
\end{equation}
Using Stirling's formula
\begin{equation*}
\lim_{x \to \infty} \frac{\Gamma(x+1)}{\sqrt{2\pi x}(x/e)^x}  = 1
\end{equation*}
it follows that $\frac{\Gamma(s)}{\Gamma(t+s)} \to 0$ as $s \to
\infty$, provided that $t > 0$. We have $t > 0$ since $a <
\frac{1}{2}$ and it follows that $\lambda_k \to 0$ as $k \to \infty$
in this case.

In the remaining case $a+\frac{d}{2} \in \mathbb{Z}$, note that $d
\geq 4$ since $a \in (1-\frac{d}{2},\frac{1}{2})$. If we let $m$ be
the integer given by
$$
m = a + \frac{d}{2} - 2
$$
then $m \in (-1,\frac{d-3}{2})$. Repeatedly using the identity
\begin{equation*}
  \Gamma(x+1) = x \Gamma(x)
\end{equation*}
it follows that
\begin{equation*}
 \bigg|\frac{\Gamma(2-a-\frac{d}{2})}{\Gamma(2-a-\frac{d}{2}-k)}\bigg| = (m+k)(m+k-1)\cdots (m+1)
\end{equation*}
and therefore
\begin{align*}
\bigg| \frac{\Gamma(2-a-\frac{d}{2})}{\Gamma(2-a-\frac{d}{2}-k)\Gamma(-a+\frac{d}{2}+k)}\bigg| & =  \frac{(m+k)(m+k-1)\cdots (m+1)}
{(-m-3+d+k)!} \\
& \leq  \frac{1}{-m-3+d+k}.
\end{align*}
It follows that $\lambda_k \to 0$ as $k \to \infty$ in this case
too.
\end{proof}

\begin{proof}[Proof of Theorem \ref{t:eigenfunctionT}]
Suppose
\begin{equation*}
f(\eta) = P(\eta) f_0(|\eta|) |\eta|^{-d/2 - k + 1/2},
\end{equation*}
where $P$ is any spherical harmonic of degree $k$, and $f_0$ is any
element of $L^2(0,\infty)$. By Theorem \ref{c:compact},
\begin{align*}
Tf(\eta) = T_\mathbb{S}P(\eta')f_0(|\eta|)|\eta|^{-d/2 + 1/2} = \lambda_k  P(\eta')f_0(|\eta|)|\eta|^{-d/2 + 1/2} = \lambda_k f(\eta),
\end{align*}
where $\eta' = |\eta|^{-1}\eta$. Theorem \ref{t:eigenfunctionT} now
follows from \eqref{e:T}.
\end{proof}

We conclude this section with several remarks on the homogeneous
weight case.
\begin{remarks}
(1) When $a=0$ one may proceed slightly differently. In this case
one can check that
$$
\lambda_k = \frac{(d-2)\pi^{d/2}}{(d+2k-2)\Gamma(\frac{1}{2}d)}
$$
so it suffices to show that
\begin{equation*}
\int_{\mathbb{S}^{d-1}} \frac{1}{|\theta - \eta'|^{d-2}}P(\theta) \, \mathrm{d}\sigma(\theta) = 2\lambda_k P(\eta')
\end{equation*}
for each nonzero $\eta$. By a limiting argument, since $\theta
\mapsto (1 - \theta \cdot \omega)^{2-d} \in L^1(\mathbb{S}^{d-1})$,
it suffices to prove that
\begin{equation} \label{e:eigensphericalRiesz}
\lim_{t \to 1} \int_{\mathbb{S}^{d-1}} \frac{1}{|\theta - t\eta'|^{d-2}}P(\theta) \, \mathrm{d}\sigma(\theta)  = 2\lambda_k P(\eta').
\end{equation}
Expanding the kernel as a power series we have
\begin{equation*}
\frac{1}{|\theta - t\omega|^{d-2}} = \frac{1}{(1 - 2(\theta \cdot \omega)t + t^2)^{(d-2)/2}} = \sum_{\ell = 0}^
\infty C_{d,\ell}(\theta \cdot \omega)t^\ell
\end{equation*}
for $|t| < 1$. Crucially, we have that the operator
$\textbf{P}_{d,\ell}$ given by
$$
\textbf{P}_{d,\ell} F (\omega) = \frac{\frac{1}{2}(d-2) + \ell}{\frac{1}{2}(d-2)|\mathbb{S}^{d-1}|} \int_{\mathbb{S}^{d-1}} C_{d,\ell}(\omega \cdot \theta)F(\theta) \, \mathrm{d}\sigma(\theta)
$$
for $F \in L^2(\mathbb{S}^{d-1})$ is the orthogonal projection from
$L^2(\mathbb{S}^{d-1})$ to the subspace of functions on
$\mathbb{S}^{d-1}$ which arise as the restriction of harmonic
polynomials of $d$ variables and homogeneous of degree $\ell$. A
proof of this fact may be found in \cite{Shimakura} (see Corollary
4.2). So
\begin{align*}
\int_{\mathbb{S}^{d-1}} \frac{1}{|\theta - t\eta'|^{d-2}}P(\theta) \, \mathrm{d}\sigma(\theta)
& = \sum_{\ell = 0}^\infty \bigg( \frac{\frac{1}{2}(d-2)|\mathbb{S}^{d-1}|}{\frac{1}{2}(d-2) + \ell} \textbf{P}_{d,\ell}P(\eta') \bigg) t^\ell \\
& = \frac{\frac{1}{2}(d-2)|\mathbb{S}^{d-1}|}{\frac{1}{2}(d-2) + k} P(\eta') t^k
\end{align*}
and \eqref{e:eigensphericalRiesz} now follows.

(2) One may show that $T_\mathbb{S}$ is compact without identifying
an explicit spectral decomposition using a more direct argument. In
particular, it suffices to show the strong operator convergence
\begin{equation} \label{e:strongnorm}
\lim_{\varepsilon \to 0} \| T_\mathbb{S} - T_\mathbb{S}^\varepsilon \| = 0,
\end{equation}
where
\begin{equation*}
T_\mathbb{S}^\varepsilon f(\omega) = \frac{1}{2} \int_{\mathbb{S}^{d-1}} \frac{1-\chi_{(0,\varepsilon)}(|\theta - \omega|)}{|\theta - \omega|^{d+2a-2}} f(\theta)
\,\mathrm{d}\sigma(\theta),
\end{equation*}
because each $T_\mathbb{S}^\varepsilon$ is compact (the kernel
$(\theta,\omega) \mapsto |\theta - \omega|^{-(d+2a-2)}(1 -
\chi_{(0,\varepsilon)})(|\theta-\omega|) \in L^2(\mathbb{S}^{d-1}
\times \mathbb{S}^{d-1})$ and compactness follows from the standard
argument for Hilbert--Schmidt kernels on bounded domains).

To see \eqref{e:strongnorm}, for each $f \in L^2(\mathbb{S}^{d-1})$,
Cauchy--Schwarz implies
\begin{align*}
| (T_\mathbb{S} - T_\mathbb{S}^\varepsilon)f(\omega)|^2 & \leq \int_{\mathbb{S}^{d-1}} \frac{\chi_{(0,\varepsilon)}(|\theta - \omega|)}{|\theta - \omega|^{d+2a-2}} \, \mathrm{d}\sigma(\theta) \int_{\mathbb{S}^{d-1}} \frac{|f(\theta)|^2}{|\theta - \omega|^{d+2a-2}} \, \mathrm{d}\sigma(\theta) \\
& = \int_{\mathbb{S}^{d-1}} \frac{\chi_{(0,\varepsilon)}(|\theta - e_1|)}{|\theta - e_1|^{d+2a-2}} \, \mathrm{d}\sigma(\theta) \int_{\mathbb{S}^{d-1}} \frac{|f(\theta)|^2}{|\theta - \omega|^{d+2a-2}} \, \mathrm{d}\sigma(\theta)
\end{align*}
so that
\begin{align*}
\| (T_\mathbb{S} - T_\mathbb{S}^\varepsilon)f \|_{L^2(\mathbb{S}^{d-1})}^2 \lesssim_{a,d}
\|f\|_{L^2(\mathbb{S}^{d-1})}^2 \int_{\mathbb{S}^{d-1}} \frac{\chi_{(0,\varepsilon)}(|\theta - e_1|)}{|\theta - e_1|^{d+2a-2}} \, \mathrm{d}\sigma(\theta).
\end{align*}
Here we have used the restriction $a \in
(1-\frac{d}{2},\frac{1}{2})$ to obtain the finiteness of the
integral
$$
\int_{\mathbb{S}^{d-1}} \frac{1}{|e_1 - \omega|^{d+2a-2}} \, \mathrm{d}\sigma(\omega).
$$
Now
\begin{align*}
\int_{\mathbb{S}^{d-1}} \frac{\chi_{(0,\varepsilon)}(|\theta - e_1|)}{|\theta - e_1|^{d+2a-2}} \, \mathrm{d}\sigma(\theta) \sim_{a,d}
\int_{1- \frac{1}{2}\varepsilon^2 < t < 1}
\frac{1}{(1-t)^{\frac{1+2a}{2}}} \, \mathrm{d}t \sim_{a,d} \varepsilon^{1-2a}
\end{align*}
and since $a \in (1-\frac{d}{2},\frac{1}{2})$ we get
\eqref{e:strongnorm}.

(3) In the homogeneous weight case $(w(r),\psi(r),\phi(r)) =
(r^{-2(1-a)},r^a,r^2)$ it is straightforward to check that
$\alpha_k(\varrho)$ is constant in $\varrho$, for each $k \in
\mathbb{N}_0$. An explicit value of this constant follows from
\begin{equation} \label{e:JL2}
\int_0^\infty J_{\nu}(r)^2 \, \frac{\mathrm{d}r}{r^\lambda} =
\frac{\Gamma(\lambda)\Gamma(\nu - \frac{1}{2}\lambda + \frac{1}{2})}{2^\lambda \Gamma(\frac{1}{2}\lambda + \frac{1}{2})^2 \Gamma(\nu + \frac{1}{2} \lambda + \frac{1}{2})},
\end{equation}
which is valid for each $0 < \lambda < 2\nu +1$. One can find
\eqref{e:JL2} in Watson \cite{Watson} (page 403, formula (2)), or
prove it directly from \eqref{e:Psigmahat}. In fact,
\begin{equation*}
\alpha_k = 2^{2(a-1)} \frac{\Gamma(1-2a) \Gamma(\nu(k) + a)}{\Gamma(1-a)^2 \Gamma(\nu(k) + 1 - a)}
\end{equation*}
and it is straightforward to check that this is decreasing in $k$.
We also remark that \eqref{e:JL2} has appeared in related work
\cite{Choetal1} and \cite{Choetal2}, where the emphasis is not on
obtaining optimal constants.
\end{remarks}

\section{Inhomogeneous weights: Proof of Theorem
\ref{t:inhomoconstants}} \label{section:inhomo}

For $k \in \mathbb{N}_0$ let $\beta_k$ be given by
\begin{equation*}
  \beta_k(\varrho) = \varrho \int_0^\infty J_{\nu(k)}(r\varrho)^2 \frac{r}{1+r^2} \, \mathrm{d}r
\end{equation*}
for $\varrho \in [0,\infty)$. The following lemma concerning the
shape of each $\beta_k$ is key to our proof of Theorem
\ref{t:inhomoconstants}. The modified Bessel functions of the first
kind, $I_\nu$ and $K_\nu$, are given by
\begin{equation*}
I_\nu(\varrho) = i^{-\nu} J_\nu(i\varrho) \quad \text{and} \quad K_\nu(\varrho) = \frac{\pi}{2\sin (\nu \pi)} (I_{-\nu}(\varrho) - I_{\nu}(\varrho)).
\end{equation*}
We shall need the following special cases
\begin{equation} \label{e:IKhalf}
I_{1/2}(r) =  (\tfrac{2}{\pi r})^{1/2} \sinh(r), \quad K_{1/2}(r) = (\tfrac{\pi}{2r})^{1/2} e^{-r}
\end{equation}
and
\begin{equation} \label{e:IK3/2}
I_{3/2}(r) = (\tfrac{2}{\pi r})^{1/2} (\cosh(r) - r^{-1}\sinh(r)),
\quad K_{3/2}(r) = (\tfrac{\pi}{2r})^{1/2}(1+r^{-1})e^{-r}.
\end{equation}
\begin{lemma} \label{l:productmodified}
For each $k \in \mathbb{N}_0$ and $\varrho \in [0,\infty)$ we have
\begin{equation} \label{e:productmodified}
\beta_k(\varrho) = \varrho I_{\nu(k)}(\varrho) K_{\nu(k)}(\varrho).
\end{equation}
Furthermore, $\beta_k$ is nonnegative, strictly concave, tends to
zero as $\varrho$ tends to zero, and tends to $\frac{1}{2}$ as
$\varrho$ tends to infinity.
\end{lemma}
\begin{proof}
The identity \eqref{e:productmodified} can be found in \cite{GR}
(page 671, formula 6.535), and the claimed limits for $\beta_k$
follow immediately from Theorem \ref{t:limits}. The strict
increasingness and concavity of $\beta_k$ follows from work of
Hartman \cite{Hartman} for $\nu(k)
> \frac{1}{2}$. This covers all $k \in \mathbb{N}_0$ and $d \geq 3$ except
for $(k,d) = (0,3)$, however a direct calculation using
\eqref{e:IKhalf} reveals that
$$
\varrho I_{\nu(0)}(\varrho)K_{\nu(0)}(\varrho) = \tfrac{1}{2} (1-
e^{-2\varrho})
$$
in this case and the desired conclusion holds in this case too. For
$\nu(k) > \frac{1}{2}$, the point is that $\varrho \mapsto
\varrho^{1/2}I_{\nu(k)}(\varrho)$ and $\varrho \mapsto
\varrho^{1/2}K_{\nu(k)}(\varrho)$ are linearly independent solutions
of
$$
x''(\varrho) - (1 + (\nu(k)^2 - \tfrac{1}{4})\varrho^{-2})x(\varrho) = 0,
$$
a special case of the Whittaker differential equation. See Theorem
4.1 of \cite{Hartman} for precisely the result that $\varrho \mapsto
\varrho I_{\nu(k)}(\varrho)K_{\nu(k)}(\varrho)$ is strictly
increasing and strictly concave on $(0,\infty)$. We also note that
earlier work of Hartman and Watson \cite{HartmanWatson} gives the
strict increasingness for all $\nu(k) \geq \frac{1}{2}$.
\end{proof}
\begin{remark}
  If $(w(r),\psi(r),\phi(r)) = ((1+r^2)^{-1},r^{1/2},r^2)$ then
  $\alpha_k(\varrho) = \frac{1}{2}\beta_k(\varrho)$. It follows
  from Lemma \ref{l:productmodified} that $\alpha = \frac{1}{4}$,
  and this shows how Theorem \ref{t:Walther} recovers the optimal
  constant in \eqref{e:inhomoSimon} (due to Simon \cite{Simon}).
\end{remark}

\begin{proof}[Proof of Theorem \ref{t:inhomoconstants}]
First, we consider the case $(w(r),\psi(r),\phi(r)) =
((1+r^2)^{-1},(1+r)^{1/2},r^2)$. By Lemma \ref{l:productmodified} it
follows that
\begin{equation*}
\alpha_k(\varrho) = \tfrac{1}{2} (1+ \varrho) I_{\nu(k)}(\varrho) K_{\nu(k)}(\varrho).
\end{equation*}
Of course, by Lemma \ref{l:productmodified} we know that $\varrho
\mapsto \varrho I_{\nu(k)}(\varrho) K_{\nu(k)}(\varrho)$ is strictly
increasing on $(0,\infty)$. However, $\varrho \mapsto
I_{\nu(k)}(\varrho) K_{\nu(k)}(\varrho)$ is strictly decreasing on
$(0,\infty)$. This fact was proved by Phillips and Malin
\cite{PhillipsMalin} when $\nu(k) \in \mathbb{N}$ and recently
Penfold, Vanden-Broeck and Grandison \cite{PVG} for all $\nu(k) \geq
0$ (see also work of Baricz \cite{BariczPAMS} who extended this to
$\nu(k) \geq -\frac{1}{2}$ with a short proof). However, we may
immediately reduce considerations to the case $k=0$ because the
function $\nu \mapsto I_\nu(\varrho) K_\nu(\varrho)$ is strictly
decreasing on $[0,\infty)$ for each fixed $\varrho > 0$ (see, for
example, \cite{BariczP}).

When $d=3$, from \eqref{e:IKhalf} we have
\begin{equation*}
\alpha_0(\varrho) = \frac{1 + \varrho}{4\varrho}(1- e^{-2\varrho})
\end{equation*}
and it is straightforward to check this is strictly decreasing for
$\varrho \in (0,\infty)$. Hence $\alpha = \alpha_0(0) = \frac{1}{2}$
in this case, or equivalently, $\mathbf{C}_3(w,\psi,\phi) =
\pi^{1/2}$ as claimed.

When $d=5$, using \eqref{e:IK3/2} we obtain
$$
\alpha_0(\varrho) = \tfrac{1}{2} \varrho^{-3} (1+\varrho)^2e^{-\varrho}(\varrho \cosh \varrho - \sinh \varrho).
$$
We claim that $\alpha_0$ has a unique global maximum on
$(0,\infty)$. To see this, note that
$$
\alpha_0'(\varrho) = \tfrac{1}{2} \varrho^{-4}(1+\varrho)e^{-\varrho} ((3+2\varrho+2\varrho^2+\varrho^3)\sinh \varrho - \varrho(3+2\varrho+\varrho^2)\cosh \varrho)
$$
and so it suffices to show that
$$
\Upsilon(\varrho) = (3+2\varrho+2\varrho^2+\varrho^3)\sinh \varrho - \varrho(3+2\varrho+\varrho^2)\cosh \varrho
$$
has a unique positive root. Now
$$
\Upsilon'(\varrho) = (\varrho - 2) (\varrho(1+\varrho)\cosh \varrho - (1+\varrho+\varrho^2)\sinh \varrho)
$$
and it is straightforward to check that $\Upsilon'(\varrho) > 0$ for
$\varrho \in (0,2)$ and $\Upsilon'(\varrho) \leq 0$ for $\varrho \in
[2,\infty)$. Since $\Upsilon(0) = 0$ and
$$
\Upsilon(\varrho) \leq (3-\varrho)\cosh \varrho < 0
$$
for $\varrho > 3$ it follows that $\Upsilon$ has a unique positive
root. It follows that $\alpha = \alpha_0(\varrho_0)$, where
$\varrho_0$ is the unique positive solution of $\Upsilon(\varrho_0)
= 0$, and hence $\mathbf{C}_5(w,\psi,\phi) = (2\pi
\alpha_0(\varrho_0))^{1/2}$ as claimed.

Now suppose $(w(r),\psi(r),\phi(r)) =
((1+r^2)^{-1},(1+r^2)^{1/4},r^2)$. Again, from monotonicity in the
index, we may reduce considerations to computing $\alpha =
\sup_{\varrho \in [0,\infty)} \alpha_0(\varrho)$. When $d=3$, we may
simply observe that $\psi(r) \leq (1 + r)^{1/2}$ and the above
considerations immediately give that $\alpha = \alpha_0(0) =
\frac{1}{2}$, or equivalently $\mathbf{C}_3(w,\psi,\phi) =
\pi^{1/2}$. When $d=5$, we have
$$
\alpha_0'(\varrho) = - \tfrac{1}{4} \varrho^{-4} (1+\varrho^2)^{-1/2}
(3+6\varrho + 6\varrho^3 + 4\varrho^4 + 2\varrho^5 -3e^{2\varrho} -
\varrho^2(e^{2\varrho}-7))
$$
and using the Maclaurin series for $e^{2\varrho}$ it follows that
$\alpha_0'(\varrho) > 0$ for all $\varrho \in (0,\infty)$.
Therefore, using Theorem \ref{t:limits},
$$
\alpha = \lim_{\varrho \to \infty} \alpha_0(\varrho) = \tfrac{1}{4}
$$
and hence $\mathbf{C}_5(w,\psi,\phi) = (\pi/2)^{1/2}$.
\end{proof}

It is now clear that \eqref{e:conj} holds for every $(w,\psi,\phi)$
considered to this point. In the homogeneous case considered in
Section \ref{section:homo}, $\sup_{\varrho \in [0,\infty)}
\alpha_0(\varrho)$ is attained everywhere since $\alpha_0$ is
constant (in fact, each $\alpha_k$ is constant in this case). For
the inhomogeneous cases considered in Theorem
\ref{t:inhomoconstants}, the supremum is attained at a unique point
(if we allow $\varrho = \infty$). We remark that other types of
``intermediate" behaviour are possible, including cases where
$\alpha_0$ is \emph{locally} constant. For an explicit (albeit
somewhat artificial) example, consider $(w(r),\psi(r),\phi(r)) =
(r^{-2}(\mu-\cos (r)),1,r^2)$, where $\mu
> 1$ is some fixed constant, and for simplicity let $d=3$. In this
case we have
$$
\alpha_k(\varrho) = \frac{\mu}{2(2k+1)} - \frac{1}{2} \int_0^\infty J_{k+1/2}(r)^2 \frac{\cos(r/\varrho)}{r} \, \mathrm{d}r,
$$
where we have made use of \eqref{e:JL2}. If we let $\Lambda$ be the
tent function given by $\Lambda(r) = (2-|r|)\chi_{[-2,2]}(r)$, then
$\Lambda = \chi_{[-1,1]} * \chi_{[-1,1]}$. Since the Fourier
transform of $r \mapsto \frac{1}{r}\sin(r)$ is $\pi\chi_{[-1,1]}$,
using the formula \eqref{e:Jhalf}, an explicit computation leads to
\begin{equation*}
\alpha_0(\varrho) = \frac{\mu}{2} - \frac{1}{4}\Lambda(1/\varrho).
\end{equation*}
Thus, $\alpha_0(\varrho)$ takes the constant value $\frac{\mu}{2}$
for $\varrho \in [0,\frac{1}{2}]$, and, for $\varrho \in
[\frac{1}{2},\infty)$ coincides with the decreasing function
$\frac{1}{2}(\mu -1) + \frac{1}{4} \varrho^{-1}$. For $k \geq 1$ we
have
\begin{equation*}
\alpha_k(\varrho) \leq \frac{\mu+1}{2(2k+1)} < \frac{\mu}{3},
\end{equation*}
where the first inequality follows by trivially estimating the
trigonometric part of the weight and \eqref{e:JL2}, and the second
is true since $\mu > 1$. Hence
\begin{equation*}
\alpha = \sup_{\varrho \in [0,\infty)} \alpha_0(\varrho) = \frac{\mu}{2}
\end{equation*}
which is attained for any $\varrho \in [0,\frac{1}{2}]$.

We conclude with the particular case with $d=3$ and
$$
(w(r),\psi(r),\phi(r)) = (\tfrac{1}{2}N\chi_{I(N)}(r),r^{1/2},r^2),
$$
where $I(N) = (1-\frac{1}{N},1+\frac{1}{N})$ and $N$ is some fixed
positive number which will be taken sufficiently large. As we will
see, this is an example where \eqref{e:conj} is not true. Note that
$\alpha \lesssim_N 1$ in this case, which follows from
\eqref{e:inhomoSimon}. Firstly, we have
\begin{equation*}
  2\pi\alpha_0(\varrho) = 1 - \tfrac{1}{4}N\varrho^{-1}(\sin(2\varrho(1+N^{-1})) - \sin(2\varrho(1-N^{-1})))
\end{equation*}
and therefore
\begin{equation*}
  \sup_{\varrho \in [0,\infty)} \alpha_0(\varrho) \leq \frac{1}{\pi},
\end{equation*}
for each $N$. We now claim that there exists $\varrho_0 > 0$ such
that, for $N$ sufficiently large,
\begin{equation} \label{e:k=1bigger}
  \alpha_1(\varrho_0) > \frac{1}{\pi},
\end{equation}
from which it is clear that \eqref{e:conj} is not true in this case.
To see this claim, first note that
\begin{equation*}
  \Xi(\varrho_0) = \varrho_0^{-1}\sin(\varrho_0) - \cos(\varrho_0) > 1
\end{equation*}
for some $\varrho_0 \in (0,\pi)$, since $\Xi(\pi) = 1$, $\Xi'(\pi) <
0$ and by smoothness considerations. Also,
\begin{equation*}
\alpha_1(\varrho_0) = \tfrac{1}{4} N\varrho_0\int_{1-N^{-1}}^{1+N^{-1}} J_{3/2}(r\varrho_0)^2r\,\mathrm{d}r
\to \tfrac{1}{2} \varrho_0 J_{3/2}(\varrho_0)^2 = \frac{1}{\pi}\Xi(\varrho_0)^2
\end{equation*}
as $N$ tends to infinity, from which \eqref{e:k=1bigger} follows.

\end{document}